\documentclass[12pt,reqno]{amsart}
\usepackage{amssymb,a4wide,xy}
\usepackage{amsmath,amsthm,amscd}
\usepackage{enumitem}
\usepackage{xcolor}
\xyoption{all}

\usepackage{multicol}
\usepackage{hyperref}

\newtheorem{theorem}[subsection]{Theorem}
\newtheorem{lemma}[subsection]{Lemma}
\newtheorem{proposition}[subsection]{Proposition}
\newtheorem{corollary}[subsection]{Corollary}

\theoremstyle{definition}
\newtheorem{remark}[subsection]{Remark}
\newtheorem{definition}[subsection]{Definition}
\newtheorem{example}[subsection]{Example}

\def\Lie{\operatorname{\textbf{\textsf{Lie}}}}
\def\Lb{\operatorname{\textbf{\textsf{Lb}}}}
\def\Set{\operatorname{\textbf{\textsf{Set}}}}

\numberwithin{equation}{section}

\theoremstyle{definition}

\def\atwo#1#2{\begin{matrix}&\ovs{#1}{\rra}\\[-3mm]&\uns{#2}{\rra} \end{matrix}}

\def\athree#1#2{\begin{matrix}\ovs{#1}{\rra}\\[-2mm]\rra\\[-2mm]\uns{#2}{\rra}\end{matrix}}

\def\al{\alpha}
\def\be{\beta}

\def\de{\delta}

\def\lam{\lambda}

\def\rra{\rightarrow}

\def\uns{\underset}
\def\ovs{\overset}

\DeclareMathOperator{\Ker}{\mathsf {Ker}}
\DeclareMathOperator{\Coker}{\mathsf {Coker}}
\DeclareMathOperator{\Img}{\mathsf {Im}}

\DeclareMathOperator{\ab}{\rm{ab}}

\def\lim{\operatorname{lim}}

\newcommand{\mm}{\mathfrak{m}}
\newcommand{\nn}{\mathfrak{n}}
\newcommand{\Gg}{\mathfrak{g}}
\newcommand{\Aa}{\mathfrak{a}}
\newcommand{\bb}{\mathfrak{b}}

\begin{document}
\title[A non-abelian exterior product and homology of Leibniz algebras]{A non-abelian exterior product and homology of Leibniz algebras}

\author[G. Donadze]{Guram Donadze$^1$}
\address{$^1$Indian Institute of Science, Education and Research Thiruvananthapuram, 695016 Thiruvananthapuram, Kerala, India}
\email{gdonad@gmail.com}
\author[X. Garc\'ia-Mart\'inez]{Xabier Garc\'ia-Mart\'inez$^2$}
\address{$^2$Department of Mathematics, Universidade de Santiago de Compostela\\
15782 Santiago de Compostela, Spain}
\email{xabier.garcia@usc.es}
\author[E. Khmaladze]{Emzar Khmaladze$^3$}
\address{$^3$A. Razmadze Mathematical Institute of Tbilisi State University,
	Tamarashvili St. 6, 0177 Tbilisi, Georgia \& University of Georgia, Kostava St. 77a, 0171 Tbilisi, Georgia  }
\email{e.khmal@gmail.com}

\begin{abstract}
We introduce a non-abelian exterior product of two crossed modules of Leibniz algebra and investigate its relation to the low dimensional Leibniz homology. Later this non-abelian exterior product is applied to the construction of eight term exact sequence in Leibniz homology. Also its relationship to the universal quadratic functor is established, which is applied to the comparison of the second Lie and Leibniz homologies of a Lie algebra.
\end{abstract}



\maketitle

\section{Introduction}\label{S:In}

Leibniz algebras were first defined in 1965 by Bloh \cite{Blo} as a non skew-symmetric analogue of Lie algebras but they became very popular when in 1993 Loday rediscovered them in \cite{Lo}. One of the main reasons that Loday had to introduced them was that in the Lie homology complex the only property of the bracket needed was the so called Leibniz identity.
Therefore, one can think about this notion as ``non-commutative'' analog of Lie algebras and study its homology. Since then, many authors have been studying them obtaining very relevant algebraic results (\cite{LoPi}, \cite{LoPi2}) and due their relations with Physics (\cite{KiWe}, \cite{Lod}) and Geometry (\cite{FeLoOn}). Many results of Lie algebras have been extended to the Leibniz case.
As an example of these generalizations, Gnedbaye \cite{Gn99} extended to Leibniz algebras the notion of non-abelian tensor product, defined by Brown and Loday in the context of groups \cite{BroLod87} and by Ellis in the context of Lie algebras \cite{Ell91}. The non-abelian tensor product was firstly introduced as a tool in homotopy theory, but it can give us nice information about central extensions and (co)homology.

The notion of non-abelian exterior product was introduced in groups by Brown and Loday \cite{BroLod84} and also extended to the Lie case by Ellis \cite{Ell91}.
The main objective of this manuscript is to give a proper generalization of this concept to Leibniz algebras. Given two ideals $\mathfrak{a}$ and $\mathfrak{b}$ of a Lie algebra $\mathfrak{g}$, the non-abelian exterior product is the quotient of the non-abelian tensor product $\mathfrak{a} \star \mathfrak{b}$ by the elements of the form $c * c$, where $c \in \mathfrak{a} \cap \mathfrak{b}$.
This makes sense in Lie theory since these are elements of the kernel of the homomorphism $\mathfrak{g} \star \mathfrak{g} \to \mathfrak{g}$,  $g\star g' \mapsto [g, g']$ for all $g, g'\in \mathfrak{g}$, but this is not true in Leibniz algebras due to the lack of antisymmetry.
Nevertheless, the non-abelian tensor product of two ideals of a Leibniz algebra, has some duplicity in the elements of the intersection, so this will be the way we can affront the problem. In fact, our definition of the non-abelian exterior product is given for crossed modules of Leibniz algebras, which is more general concept than ideals of Leibniz algebras.

The paper is organized as follows. In Section \ref{S:Pre} we recall some basic definitions and properties of Leibniz algebras and Leibniz homology. In Section \ref{S:Ten} we introduce the non-abelian exterior product and we study its connections with the second Leibniz homology. In Section \ref{S:H3} we obtain an eight term exact sequence using the non-abelian exterior product. In Section \ref{S:quad} we explore the relations with Whitehead's universal quadratic functor. Finally, in Section \ref{S:h2} we compare the second Leibniz homology and the second Lie homology  of a Lie algebra.

\section{Leibniz algebras and homology}\label{S:Pre}

Throughout the paper $\mathbb{K}$ is a field, unless otherwise stated. All vector spaces and algebras are considered over $\mathbb{K}$, all linear maps are $\mathbb{K}$-linear maps and $\otimes$ stands for $\otimes_{\mathbb{K}}$.

\begin{definition}[\cite{Lo}]
A Leibniz algebra is a vector space $\mathfrak{g}$  equipped with a linear
map (Leibniz bracket)
\[
[\;,\; ]\colon \mathfrak{g}\otimes \mathfrak{g} \to \mathfrak{g}
\]
 satisfying the Leibniz identity
 \[
[x,[y,z]] = [[x,y],z] - [[x,z],y],
\]
for all $x,y,z \in \mathfrak{g}$.
\end{definition}

A homomorphism of Leibniz algebras is a linear map preserving the bracket. The respective category of
Leibniz algebras will be denoted by $\Lb$.

A subspace $\mathfrak{a}$ of a Leibniz algebra $\mathfrak{g}$  is called (two-sided) \emph{ideal} of $\mathfrak{g}$ if $[a,x],\;  [x,a]\in \mathfrak{a}$ for all $a\in \mathfrak{a}$ and $x\in \mathfrak{g}$. In this case the quotient space $\mathfrak{g}/\mathfrak{a}$ naturally inherits a Leibniz algebra structure.

An example of ideal of a Leibniz algebra $\mathfrak{g}$ is the \emph{commutator} of $\mathfrak{g}$, denoted by $[\mathfrak{g},\mathfrak{g}]$, which is the subspace of $\mathfrak{g}$  spanned by elements of the form $[x,y]$, $x,y\in \mathfrak{g}$. The quotient $\mathfrak{g}/[\mathfrak{g}, \mathfrak{g}]$ is denoted by $\mathfrak{g}^{\ab}$ and is called \emph{abelianization} of $\mathfrak{g}$. One more example of an ideal is the \emph{center} $C(\mathfrak{g})= \{c \in {\mathfrak{g}} \mid [x,c] = 0=[c,x],\ \text{for all} \  x \in {\mathfrak{g}}\}$ of $\mathfrak{g}$. Note that both  $\mathfrak{g}^{\ab}$ and $C(\mathfrak{g})$ are \emph{abelian Leibniz algebras}, that is, Leibniz algebras with the trivial Leibniz bracket $[\;,\;]=0$.

Clearly, any Lie algebra is a Leibniz algebra and conversely, any Leibniz algebra with the antisymmetric Leibniz bracket is a Lie algebra. This is why Leibniz algebras are called non-commutative generalization of Lie algebras. Thus, there is a full embedding functor $\Lie \hookrightarrow \Lb$, where $\Lie$ denotes the category of Lie algebras. This embedding has a left adjoint $\mathfrak{{Lie}}\colon \Lb \to \Lie$, called the Liezation functor and defined as follows. Given a Leibniz algebra $\mathfrak{g}$, $\mathfrak{Lie}(\mathfrak{g})$  is the quotient of $\mathfrak{g}$ by the subspace (that automatically is an ideal) spanned by elements of the form $[x,x]$, $x\in \mathfrak{g}$ (see for example \cite{KuPi}).

The original reason for introduction of Leibniz algebras was a new variant of Lie homology, called non-commutative Lie homology or Leibniz homology, developed in \cite{Lo0, LoPi} and denoted by $HL_*$. Let us recall the definition of $HL_*$.

Given a Leibniz algebra $\mathfrak{g}$, consider the following chain complex:
\begin{align*}
CL_*(\mathfrak{g})\equiv  \;\; \cdots \overset{d}{\longrightarrow} \mathfrak{g}^{\otimes n} \overset{d}{\longrightarrow}  \mathfrak{g}^{\otimes n-1}\overset{d}{\longrightarrow}
\cdots \overset{d}{\longrightarrow} \mathfrak{g} \overset{d}{\longrightarrow} \mathbb{K},
\end{align*}
where the boundary map $d$ is given by

\begin{align*}
&d(x)=0, \;\;\text{for each}\;\;x\in \mathfrak{g};\\
&d(x_1\otimes \cdots \otimes x_n)= \underset{1\leq i<j \leq n}{\sum} (-1)^j (x_1\otimes \cdots \otimes x_{i-1}\otimes [x_i, x_j]\otimes x_{i+1}
\otimes \cdots \otimes \widehat{x_j} \otimes \cdots \otimes x_n),
\end{align*}
for $x_1, \ldots, x_n\in \mathfrak{g}$ and $n>1$. The $n$-th homology group of $\mathfrak{g}$ is defined by
\[
HL_n(\mathfrak{g})=H_n (CL_*(\mathfrak{g})), \;\; n\geq 0.
\]

\subsection{Homology via derived functors.}\label{free}
Let $X$ be a set and $M_X$ be the free magma on $X$ with a binary operation $[\;, \;]$. Denote by $\mathbb{K}(M_X)$ the free
vector space over the set $M_X$. In a natural way we can extend $[\;, \;]$
to a binary operation $[\;, \;]$ on $\mathbb{K}(M_X)$. Thus $\mathbb{K}(M_X)$ is the free algebra on $X$ (see e.g. \cite{Se}) . For any subset $S\subseteq \mathbb{K}(M_X)$, let $N^1(S)$ denote the
vector subspace of $\mathbb{K}(M_X)$ generated by $S$ and $\{[x, s], [s, x] \mid x\in M_X, s\in S \}$.
Let $\mathcal{N}(S)=\underset{{i\geq 1}}{\cup}N^{i}(S)$ where for $i > 1$, $N^i(S) = N^1(N^{i-1}(S))$. In particular, consider
$S_X= \{[x, [y, z]]-[[x,y], z]+[[x,z],y] \mid x, y, z\in M_X\}$ and denote
$ \mathbb{K}(M_X)/\mathcal{N}(S_X)$ by $\mathfrak{F}(X)$. In other words, $\mathfrak{F}(X)$ is the quotient of $\mathbb{K}(M_X)$ by the two-sided ideal generated by the subset $S_X$. Clearly $\mathfrak{F}(X)$ is a Leibniz algebra, called the free Leibniz algebra
on the set $X$. It is easy to see that the construction $X\mapsto \mathfrak{F}(X)$ defines a covariant functor
\[
\mathfrak{F}\colon \Set \to \Lb, %\;\; X\mapsto \mathfrak{F}(X),
\]
which is left adjoint to the natural forgetful functor $\mathfrak{U} \colon  \Lb\to \Set$, where $\Set$ denote the category of sets.

It is well known that every adjoint pair of functors induces a cotriple (see for example \cite[Chapter 1]{BaBe} or \cite[Chapter 2]{Ina97}).
Let $\mathbb{F}=(\mathbb{F}, \de, \tau)$ denote the cotriple in $\Lb$ defined by the adjunction $\mathfrak{F} \dashv
\mathfrak{U}$, that is, $\mathbb{F}=\mathfrak{F}\mathfrak{U}\colon \Lb\to \Lb$, $\tau\colon \mathbb{F}\to 1_{\Lb}$ is the counit and $\de =\mathfrak{F} u \mathfrak{U}\colon\mathbb{F}\to \mathbb{F}^2$, where $u\colon 1_{\Set}\to \mathfrak{U}\mathfrak{F}$ is the unit of the adjunction.
Then, given an endofunctor $\mathfrak{T}\colon\Lb\to \Lb$, one can consider derived functors $\mathcal{L}_n^{\mathbb{F}}\mathfrak{T} \colon \Lb\to \Lb$, $n\ge 0$, with respect to the cotriple $\mathbb{F}$ (see again \cite{BaBe}). In particular, Leibniz homology can be described in terms of the derived functors of the  abelianization functor $\mathfrak{Ab} \colon \Lb \to \Lb$, defined by $\mathfrak{Ab}({\mathfrak g})=\mathfrak{g}^{\ab}$.

\begin{theorem}\label{theorem} Let $\mathfrak{g}$ be a Leibniz algebra.
Then there is an isomorphism
$$
HL_{n+1}(\mathfrak{g})\cong \mathcal{L}_n^{\mathbb{F}}\mathfrak{Ab}(\mathfrak{g}), \;\; n\geq 0.
$$
\end{theorem}
\begin{proof} Let $\mathfrak{f}_*$ be the $\mathbb{F}$-cotriple simplicial resolution of $\mathfrak{g}$, that is, $\mathfrak{f}_* = (\mathfrak{f}_*(\mathfrak{g}), d^n_i, s^n_i )$ is the simplicial Leibniz algebra with
\begin{align*}
& {\mathfrak{f}_n(\mathfrak{g})}={\mathbb{F}^{n+1}}(\mathfrak{g})={\mathbb{F}}\big({\mathbb{F}^n}(\mathfrak{g})\big),\\
& d^n_i={\mathbb{F}^i}(\tau_{{\mathbb{F}^{n-i}}}),\quad
s^n_i={\mathbb{F}^i}(\de_{{\mathbb{F}^{n-i}}}),\quad 0\le i\le n, \quad n\geq 0.
\end{align*}

Applying the functor $CL_*$ to $\mathfrak{f}_*$ dimension-wise, we get the following bicomplex:
\[
\CD
  & & \vdots   & & \vdots  & & \vdots \\
  & & @V dVV @V-dVV @VdVV \\
  & & \mathfrak{f}_0^{\otimes 2} @<{}<< \mathfrak{f}_1^{\otimes 2} @<{}<< \mathfrak{f}_2^{\otimes 2} @<{}<< \cdots \\
  & & @V dVV @V-dVV @VdVV \\
  & & \mathfrak{f}_0 @<{}<< \mathfrak{f}_1 @<{}<< \mathfrak{f}_2 @<{}<< \cdots \\
  & & @V dVV @V-dVV @VdVV \\
  & & \mathbb{K} @<{}<< \mathbb{K} @<{} << \mathbb{K} @<{}<< \cdots ,
\endCD
\]
where the horizontal differentials are obtained by alternating sums of face homomorphisms. Since $\mathbb{K}$ is a field,
$\mathfrak{f}_*^{\otimes n}\to \mathfrak{g}^{\otimes n}$ is an aspherical augmented  simplicial vector space (see for example
\cite[Lemma 2.3]{DIKL12} or \cite[Lemma 2.1]{DIL10}), for each $n\geq 1$. Therefore, we have an isomorphism:
\[
HL_n(\mathfrak{g})\cong H_n(Tot (CL_*(\mathfrak{f}_*))), \;\; n\geq 0.
\]
On the other hand we have the following spectral sequence:
$$
E^1_{pq}=H_q(\mathfrak{f}_p)\Rightarrow H_{p+q}(Tot (CL_*(\mathfrak{f}_*))) .
$$
Since $\mathfrak{f}_n$ is a free Leibniz algebra for each $n\geq 0$, we obtain:
\[
E^1_{pq}=\left\lbrace
\begin{matrix}
\mathbb{K}\qquad  & \text{for} & q = 0\;, \\
\mathfrak{f}_p/[\mathfrak{f}_p, \mathfrak{f}_p] \qquad& \text{for} & q = 1\;,\\
0 \qquad& \text{for} & q > 1\;.
\end{matrix}
\right.
\]
This spectral sequence degenerates at $E^2$ and
\[
E^2_{pq}=\left\lbrace
\begin{matrix}
\mathbb{K}\qquad  & \text{for} & q = 0& \text{and} & p = 0\;, \\
0\qquad  & \text{for} & q = 0& \text{and} & p > 0\;, \\
H_p(\mathfrak{Ab}(\mathfrak{f}_*)) \qquad& \text{for} & q = 1\;,\\
0 \qquad& \text{for} & q > 1\;.
\end{matrix}
\right.
\]
Thus, the spectral sequence argument completes the proof.
\end{proof}

\begin{remark}\label{remark1}
The $\mathbb{F}$-cotriple simplicial resolution of a Leibniz algebra $\mathfrak{g}$ is a free (projective) simplicial resolution of $\mathfrak{g}$ and by \cite[5.3]{BaBe}, if $\mathfrak{f}_*$ is any of them, then there are natural isomorphisms
\[
HL_n(\mathfrak{g})\cong \pi_{n-1}(\mathfrak{Ab}(\mathfrak{f}_*)), \quad n\geq 1.
\]
\end{remark}

\subsection{Hopf formulas.}  Theorem \ref{theorem} enables us to prove Hopf formulas for the (higher) homology of Leibniz algebras, pursuing the line and  technique developed in \cite{DIP05} for the description of higher group homology via Hopf formulas (see \cite{BroEll88}).
In this respect herewith we state two theorems without proofs. In fact they are particular cases of \cite[Theorem 15]{CaKhLa} describing homology of Leibniz $n$-algebras via Hopf formulas. Note also that, of course, these results agree with the categorical approach to the problem given in \cite{EvGrTim} for semi-abelian homology.

An extension of Leibniz algebras $0 \to \mathfrak{r}\to \mathfrak{f}\to \mathfrak{g}\to 0$ is said to be a free presentation of $\mathfrak{g}$,
if $\mathfrak{f}$ is a free Leibniz algebra over a set.

\begin{theorem}\label{theorem_Hopf2} Let $0 \to \mathfrak{r}\to \mathfrak{f}\to \mathfrak{g}\to 0$ be a free presentation of
a Leibniz algebra $\mathfrak{g}$. Then there is an isomorphism:
\[
HL_2(\mathfrak{g})\cong \frac{\mathfrak{r}\cap [\mathfrak{f}, \mathfrak{f}]}{[\mathfrak{r}, \mathfrak{f}]}.
\]
\end{theorem}

\begin{theorem}\label{theorem_Hopf3} Let $\mathfrak{r}$ and $\mathfrak{s}$ be ideals of a free Leibniz algebra $\mathfrak{f}$.
Suppose that $\mathfrak{f}/\mathfrak{r}$ and $\mathfrak{f}/\mathfrak{s}$ are free Leibniz algebras, and that
$\mathfrak{g}=\mathfrak{f}/(\mathfrak{r}+\mathfrak{s})$. Then there is an isomorphism:
\[
HL_3(\mathfrak{g})\cong \frac{\mathfrak{r}\cap\mathfrak{s}\cap [\mathfrak{f}, \mathfrak{f}]}{[\mathfrak{r}, \mathfrak{s}]+
[\mathfrak{r}\cap \mathfrak{s}, \mathfrak{f}]}.
\]
\end{theorem}

\section{Non-abelian tensor and exterior product of Leibniz algebras}\label{S:Ten}

\subsection{Leibniz actions and crossed modules.}
\begin{definition}
Let $\mm$ and $\nn$ be Leibniz algebras. A \emph{Leibniz action} of $\mm$ on
$\nn$ is a couple of bilinear maps $\mm \times \nn \to \nn$, $(m, n)\mapsto \;^mn$, and
$\nn \times \mm \to \nn$, $(n, m)\mapsto n^m$, satisfying the following axioms:
\begin{align*}
^{[m, m']}n&=\; ^m(^{m'}n)+(^mn)^{m'},& ^m[n, n']&=[^mn, n']-[^mn', n],\\
n^{[m, m']}&=(n^m)^{m'}-(n^{m'})^{m}, & [n, n']^m&=[n^m, n']+[n, n'^m],\\
^m(^{m'}n)&=-^m(n^{m'}), & [n, \;^mn']&=-[n, n'^{m}],
\end{align*}
for each $m, m'\in \mm$, $n, n'\in \nn$. For example, if $\mm$  is a subalgebra of a Leibniz algebra $\mathfrak{g}$  (maybe
  $\mathfrak{g}= \mathfrak{m}$) and $\nn$ is an ideal of $\mathfrak{g}$, then Leibniz bracket in $\mathfrak{g}$ yields a Leibniz action of $\mm$ on $\nn$.
\end{definition}

%\begin{definition}[\cite{Gn99}]
%	Let $\mm$ and $\nn$ be two Leibniz algebras acting on each other. We say these actions are \emph{compatible} if
%	\begin{align*}
%	^{(^mn)}m' &= [m^n, m'], & ^{(^nm)}n' &= [n^m, n'], \\
%	^{(n^m)}m' &= [^nm, m'], & ^{(m^n)}n' &= [^mn, n'], \\
%	m^{(^{m'}n)} &= [m, m'^n] & n^{(^{n'}m)} &= [n, n'^{m}], \\
%	m^{(n^{m'})} &=[m, \;^nm'] & n^{(m^{n'})} &= [n, \;^mn'],
%	\end{align*}
%	for $m, m' \in \mm$ and $n, n' \in \nn$.
%\end{definition}

\begin{definition}
A \emph{Leibniz crossed module} $(\mm, \Gg, \eta)$ is a homomorphism of Leibniz algebras $\eta \colon \mathfrak{m} \to \mathfrak{g}$ together with an action of $\mathfrak{g}$ on $\mathfrak{m}$ such that
\begin{align*}
	& \eta(^xm)=[g,\eta(m)],   \qquad \eta(m^x)=[\eta(m),x],\\
	& ^{\eta(m_1)}m_2=[m_1,m_2]=m_1^{\eta (m_2)},
\end{align*}
where $x \in \Gg$ and $m, m_1, m_2 \in \mm$.
\end{definition}

\begin{example}\label{E:CM}
	Let $\mathfrak{a}$ be an ideal of a Leibniz algebra $\Gg$. Then the inclusion $i \colon \Aa \to \Gg$ is a crossed module where the action of $\Gg$ on $\Aa$ is given by the bracket in $\Gg$. In particular, a Leibniz algebra can be seen as a crossed module by $1_{\Gg} \colon \Gg \to \Gg$.
\end{example}

\subsection{Non-abelian tensor product.}\label{S:tensor} Let $\mathfrak{m}$ and $\mathfrak{n}$ be Leibniz algebras with mutual actions on one another.
The \emph{non-abelian tensor product} of $\mathfrak{m}$ and $\mathfrak{n}$, denoted by $\mathfrak{m}\star \mathfrak{n}$, is defined in \cite{Gn99} to be
the Leibniz algebra generated by the symbols $m*n$ and $n*m$, for all $m\in \mathfrak{m}$ and $n\in \mathfrak{n}$, subject to the
following relations:
\begin{align*}
(1\textrm{a})\quad&k(m*n)=km*n=m*kn, & (1\textrm{b})\quad&k(n*m)=kn*m=n*km,\\
(2\textrm{a})\quad&(m+m')*n=m*n+m'*n, &  (2\textrm{b})\quad&(n+n')*m=n*m+n'*m,\\
(2\textrm{c})\quad&m*(n+n')=m*n+m*n', & (2\textrm{d})\quad&n*(m+m')=n*m+n*m',\\
(3\textrm{a})\quad&m*[n, n']=m^n*n'-m^{n'}*n, & (3\textrm{b})\quad&n*[m, m']=n^m*m'-n^{m'}*m,\\
(3\textrm{c})\quad&[m, m']*n=\;^mn*m'-m*n^{m'}, & (3\textrm{d})\quad&[n, n']*m=\;^nm*n'-n*m^{n'},\\
(4\textrm{a})\quad&m*\;^{m'}n=-m*n^{m'}, & (4\textrm{b})\quad&n*\;^{n'}m=-n*m^{n'},\\
(5\textrm{a})\quad&m^n*\;^{m'}n'=[m*n, m'*n']=\;^mn*m'^{n'}, & (5\textrm{b})\quad&^nm*n'^{m'}=[n*m, n'*m']=n^m*\;^{n'}m',\\
(5\textrm{c})\quad&m^n*n'^{m'}=[m*n, n'*m']=\;^mn*\;^{n'}m', & (5\textrm{d})\quad&^nm*\;^{m'}n'=[n*m, m'*n']=n^m* m'^{n'},
\end{align*}
for each $k\in \mathbb{K}$, $m, m'\in \mathfrak{m}$, $n, n'\in \mathfrak{n}$.

There are induced homomorphisms of Leibniz algebras $\tau_{\mm} \colon \mm \star \nn \to \mm$ and $\tau_{\nn} \colon \mm \star \nn \to \nn$ where $\tau_{\mm}(m * n) = m^n$, $\tau_{\mm}(n * m) = \;^nm$, $\tau_{\nn}(m * n) = \;^mn$ and $\tau_{\nn}(n * m) = n^m$.

%\begin{align*}
%k(x*y)&=kx*y=x*ky, & k(y*x)&=ky*x=y*kx,\\
%(x+x')*y&=x*y+x'*y, & (y+y')*x&=y*x+y'*x,\\
%x*(y+y')&=x*y+x*y', & y*(x+x')&=y*x+y*x',\\
%x*[y, y']&=x^y*y'-x^{y'}*y, & y*[x, x']&=y^x*x'-y^{x'}*x,\\
%[x, x']*y&=\;^xy*x'-x*y^{x'}, & [y, y']*x&=\;^yx*y'-y*x^{y'},\\
%x*\;^{x'}y&=-x*y^{x'}, & y*\;^{y'}x&=-y*x^{y'},\\
%x^y*\;^{x'}y'&=[x*y, x'*y']=\;^xy*x'^{y'},\\
%x^y*y'^{x'}&=[x*y, y'*x']=\;^xy*\;^{y'}x',\\
%^yx*y'^{x'}&=[y*x, y'*x']=y^x*\;^{y'}x',\\
%^yx*\;^{x'}y'&=[y*x, x'*y']=y^x* x'^{y'},
%\end{align*}
%\begin{multicols}{2}
%\begin{enumerate}
%\item[(1a)] $k(x*y)=kx*y=x*ky$,
%\item[(2a)] $(x+x')*y=x*y+x'*y$,
%\item[(2c)] $x*(y+y')=x*y+x*y'$,
%\item[(3a)] $x*[y, y']=x^y*y'-x^{y'}*y$,
%\item[(3c)] $[x, x']*y=\;^xy*x'-x*y^{x'}$,
%\item[(4a)] $x*\;^{x'}y=-x*y^{x'}$,
%\item[(5a)] $x^y*\;^{x'}y'=[x*y, x'*y']=\;^xy*x'^{y'},$
%\item[(5c)] $x^y*y'^{x'}=[x*y, y'*x']=\;^xy*\;^{y'}x',$
%\item[(1b)] $k(y*x)=ky*x=y*kx,$
%\item[(2b)] $(y+y')*x=y*x+y'*x,$
%\item[(2d)] $y*(x+x')=y*x+y*x',$
%\item[(3b)] $y*[x, x']=y^x*x'-y^{x'}*x,$
%\item[(3d)] $[y, y']*x=\;^yx*y'-y*x^{y'},$
%\item[(4b)] $y*\;^{y'}x=-y*x^{y'},$
%\item[(5b)] $^yx*y'^{x'}=[y*x, y'*x']=y^x*\;^{y'}x',$
%\item[(5d)] $^yx*\;^{x'}y'=[y*x, x'*y']=y^x* x'^{y'},$
%\end{enumerate}
%\end{multicols}

\subsection{Non-abelian exterior product.}\label{exterior}
Let us consider two Leibniz crossed modules $\eta \colon \mm \to \Gg$ and $\mu \colon \nn \to \Gg$. Then there are induced actions of $\mm$ and $\nn$ on each other via the action of $\Gg$. Therefore, we can consider the non-abelian tensor product $\mm \star \nn$. We define $\mm \square \nn$ as the vector subspace of $\mm \star \nn$ generated by the elements $m * n' - n * m'$ such that $\eta(m) = \mu(n)$ and $\eta(m') = \mu(n')$.
\begin{proposition}\label{P:cen}
	The vector subspace $\mm \square \nn$ is contained in the center of $\mm \star \nn$, so in particular it is an ideal of $\mm \star \nn$.
\end{proposition}
\begin{proof}
	Everything can be readily checked by using defining relations  (5\textrm{a})-(5\textrm{d}) of $\mm \star \nn$. For instance, for any $m''\in \mm$ and $n''\in \nn$, we have
	\begin{align*}
	[m * n' - n * m', m'' * n''] &= m^{n'} * \; ^{m''}n'' - \; ^nm' * \; ^{m''}n'' \\
	{} &= m^{\mu(n')} * \; ^{m''}n'' - \; ^{\mu(n)}m' * \; ^{m''}n'' \\
	{} &= m^{\eta(m)} * \; ^{m''}n'' - \; ^{\eta(m)}m' * \; ^{m''}n'' \\
	{} &= [m, m'] * \; ^{m''}n'' - [m, m'] * \; ^{m''}n'' \\
	{} &= 0.
	\end{align*}
\end{proof}

\begin{definition}
	Let $\eta \colon \mm \to \Gg$ and $\mu \colon \nn \to \Gg$ be two Leibniz crossed modules in the previous setting. We define the \emph{non-abelian exterior product} $\mm \curlywedge \nn$ of $\mm$ and $\nn$ by
	\[
	\mm \curlywedge \nn = \dfrac{\mm \star \nn} {\mm \square \nn}.
	\]
	The cosets of $m * n$ and $n * m$ will be denoted by $m \curlywedge n$ and $n \curlywedge m$, respectively.
\end{definition}

\begin{remark}\label{R: ring} Definitions of the non-abelian tensor and exterior products do not require $\mathbb{K}$ to be necessarily a field.
It is clear that one can do the same for a commutative ring with identity.
\end{remark}

There is an epimorphism of Leibniz algebras $\pi\colon \mathfrak{m}\star \mathfrak{n} \to \mathfrak{m}\curlywedge \mathfrak{n}$ sending $m * n$ and $n * m$ to $m \curlywedge n$ and $n \curlywedge m$, respectively.

To avoid any confusion, let us note that, given a Leibniz algebra $\mathfrak{g}$, for each $x,y\in \mathfrak{g}$, the non-abelian tensor square $\mathfrak{g}\star \mathfrak{g}$ has two copies of generators of the form $x *y$,  and exactly these generators are identified in the non-abelian exterior square $\mathfrak{g}\curlywedge \mathfrak{g}$. Thus we need to distinguish two copies of inclusions (identity maps) $i_1, i_2\colon \mathfrak{g}\to \mathfrak{g}$, $i_i=i_2=1_{\mathfrak{g}}$ and take $\mathfrak{g}\curlywedge \mathfrak{g}$ to be the quotient of $\mathfrak{g}\star \mathfrak{g}$ by the relation
$i_1(x)*i_2(y)=i_2(x)*i_1(y)$  for each $x, y\in \mathfrak{g}$.

In the case of $\mathfrak{a}$ and $\mathfrak{b}$ being two ideals of a Leibniz algebra $\mathfrak{g}$ seen as crossed modules, the non-abelian exterior product $\Aa \curlywedge \bb$ is just $\Aa \star \bb$ quotient by the elements of the form $i_1(c) \curlywedge i_2(c') - i_2(c) \curlywedge i_1(c')$, where $c, c' \in \mathfrak{a}\cap \mathfrak{b}$; $i_1\colon \mathfrak{a}\cap \mathfrak{b} \to \mathfrak{a}$ and $i_2\colon \mathfrak{a}\cap \mathfrak{b} \to \mathfrak{b}$
are the natural inclusions.

The proof of the following proposition is immediate.
\begin{proposition}
Let $\mathfrak{a}$ and $\mathfrak{b}$ be two ideals of a Leibniz algebra. There is a homomorphism of Leibniz algebras
\[
\theta_{\mathfrak{a}, \mathfrak{b}} \colon
\mathfrak{a}\curlywedge \mathfrak{b} \to \mathfrak{a}\cap \mathfrak{b}
\]
 defined on generators by
$\theta_{\mathfrak{a}, \mathfrak{b}}(a\curlywedge b)= [a, b]$ and $\theta_{\mathfrak{a}, \mathfrak{b}}(b\curlywedge a)=[b, a]$,
for all $a\in \mathfrak{a}$ and $b\in \mathfrak{b}$. Moreover, $\theta_{\mathfrak{a}, \mathfrak{b}}$ is a crossed module of Leibniz algebras (c.f. \cite[Proposition 4.3]{Gn99}).
\end{proposition}

\begin{proposition}
 Let $\mathfrak{g}$ be a perfect Leibniz algebra, that is, $\mathfrak{g}=[\mathfrak{g},\mathfrak{g}]$. Then
 the homomorphism $\theta_{\mathfrak{g}, \mathfrak{g}} \colon \mathfrak{g}\curlywedge \mathfrak{g} \twoheadrightarrow \mathfrak{g}$ is the universal central extension of $\mathfrak{g}$.
\end{proposition}
\begin{proof}The last four identities of the non-abelian tensor product immediately imply that $\mathfrak{g}\star\mathfrak{g}=\mathfrak{g}\curlywedge \mathfrak{g}$. Hence,  by \cite[Theorem 6.5]{Gn99} the homomorphism $\theta_{\mathfrak{g}, \mathfrak{g}} \colon \mathfrak{g}\curlywedge \mathfrak{g} \twoheadrightarrow \mathfrak{g}$ is the universal central extension of the perfect Leibniz algebra $\mathfrak{g}$.
\end{proof}

\subsection{Relationship to the second homology}

Let $\gamma\colon\mathfrak{g}\to \mathfrak{h}$ be a homomorphism of Leibniz algebras, $\mathfrak{a}$ and $\mathfrak{a}'$ (resp. $\mathfrak{b}$ and $\mathfrak{b}'$) be two ideals of $\mathfrak{g}$ (resp. $\mathfrak{h}$) such that $\gamma(\mathfrak{a})\subseteq \mathfrak{b}$ and $\gamma(\mathfrak{a}')\subseteq \mathfrak{b}'$. Since $\gamma(\mathfrak{a}\cap \mathfrak{a}')\subseteq \mathfrak{b}\cap \mathfrak{b}'$, it is easy to see that $\gamma$ induces a homomorphism of Leibniz algebras $\mathfrak{a}\curlywedge \mathfrak{a}'\to \mathfrak{b}\curlywedge \mathfrak{b}'$ in the natural way: $a \curlywedge a'\mapsto \gamma(a)\curlywedge \gamma(a')$ and $a' \curlywedge a\mapsto \gamma(a')\curlywedge \gamma(a)$, for all $a\in \mathfrak{a}$, $a'\in \mathfrak{a}'$.

Now suppose that $0\to \mathfrak{a}\to \mathfrak{g}\to \mathfrak{h} \to 0$ and $0\to \mathfrak{a}'\to \mathfrak{g}'\to \mathfrak{h}' \to 0$
are extensions of Leibniz algebras, where $\mathfrak{a}'$ and $\mathfrak{g}'$ are ideals of $\mathfrak{g}$, while $\mathfrak{h}'$ is an ideal of $\mathfrak{h}$. Then the following naturally induced map $\mathfrak{g}\curlywedge \mathfrak{a}' \times \mathfrak{a}\curlywedge \mathfrak{g}' \to \mathfrak{g}\curlywedge \mathfrak{g}'$
is not in general a homomorphism of Leibniz algebras, but the following sequence
\begin{align}\label{sequence1}
\mathfrak{g}\curlywedge \mathfrak{a}' \times \mathfrak{a}\curlywedge \mathfrak{g}' \to \mathfrak{g}\curlywedge \mathfrak{g}'
\to \mathfrak{h}\curlywedge \mathfrak{h}' \to 0
\end{align}
is exact, in the sense that
$\Img \big(\mathfrak{g}\curlywedge \mathfrak{a}' \times \mathfrak{a}\curlywedge \mathfrak{g}' \to \mathfrak{g}\curlywedge \mathfrak{g}' \big)=
\Ker \big( \mathfrak{g}\curlywedge \mathfrak{g}' \to \mathfrak{h}\curlywedge \mathfrak{h}'\big)$.

\begin{lemma} \label{lemma1} Let $0\to \mathfrak{a}\to \mathfrak{g}\to \mathfrak{h} \to 0$ be an extension of Leibniz algebras.
Then, the following induced sequence of Leibniz algebras $\mathfrak{a}\curlywedge \mathfrak{g}\to \mathfrak{g}\curlywedge \mathfrak{g}\to
\mathfrak{h} \curlywedge \mathfrak{h}\to 0$ is exact.
\end{lemma}
\begin{proof} Since the images of the induced homomorphisms $\mathfrak{a}\curlywedge \mathfrak{g}\to \mathfrak{g}\curlywedge \mathfrak{g}$ and $\mathfrak{g}\curlywedge \mathfrak{a}\to \mathfrak{g}\curlywedge \mathfrak{g}$ are the same, the statement follows immediately from the exactness of (\ref{sequence1}).
\end{proof}

\begin{proposition}\label{proposition1} Let $\mathfrak{f}$ be a free Leibniz algebra over a set $X$.
Then $\theta_{\mathfrak{f}, \mathfrak{f}}$ is injective.
\end{proposition}
\begin{proof} We will consider $\theta_{\mathfrak{f}, \mathfrak{f}}$ as an epimorphism $\mathfrak{f}\curlywedge \mathfrak{f} \twoheadrightarrow [\mathfrak{f}, \mathfrak{f}]$ and show that it is an isomorphism.
Using the same notations as in Subsection \ref{free}, suppose $[M_X, M_X]$ denotes the subset $\{[x, y] \mid x, y\in M_X\}$ of $M_X$ and
$\mathbb{K}[M_X, M_X]$ denotes the free vector space over the set $[M_X, M_X]$. Then,
\[
[\mathfrak{f}, \mathfrak{f}]=\frac{\mathbb{K}[M_X, M_X]}{\mathcal{N}(S_X)}.
\]
Note that for each element $m\in [M_X, M_X]$ there are unique $x$ and $y$ in $M_X$ such that $m=[x, y]$.
Therefore, the following map $\tau \colon \mathbb{K}[
M_X, M_X] \to \mathfrak{f}\curlywedge \mathfrak{f}$, given by
$[x, y]\mapsto x \curlywedge y$ for each $x, y\in M(X)$, is well-defined. We have
$$
[x, [y, z]]-[[x,y], z]+[[x,z],y]\overset{\tau}{\mapsto} x\curlywedge [y, z]-[x,y]\curlywedge z+[x,z]\curlywedge y= 0,
$$
for each $x, y, z\in M_X$. Moreover, if $m=\underset{i=1}{\overset{n}{\sum}} k_i x_i \in S_X$ with $k_1, \ldots, k_n\in \mathbb{K}$
and $x_1, \ldots, x_n\in M_X$, then we have:
\[
[x, m]\overset{\tau}{\mapsto} \overset{n}{\underset{i=1}{\sum}}k_i (x\curlywedge x_i)=x\curlywedge \Big(\overset{n}{\underset{i=1}{\sum}}k_i x_i\Big)=0,
\]
\[
[m, x]\overset{\tau}{\mapsto} \overset{n}{\underset{i=1}{\sum}}k_i (x_i\curlywedge x)=\Big(\overset{n}{\underset{i=1}{\sum}}k_i x_i\Big)\curlywedge x =0,
\]
for each $x\in M_X$. As a result we have that $\tau (\mathcal{N}(S_X))=0$. Thus, $\tau$ induces a well-defined
linear map $\tau^*\colon [\mathfrak{f}, \mathfrak{f}]\to \mathfrak{f}\curlywedge \mathfrak{f}$. Furthermore,
$\tau^* \circ {\theta_{\mathfrak{f}, \mathfrak{f}}}=1_{\mathfrak{f}\curlywedge \mathfrak{f}}$ and
${\theta_{\mathfrak{f}, \mathfrak{f}}}\circ\tau^* =1_{[\mathfrak{f}, \mathfrak{f}]}$. This completes the proof.
\end{proof}

\begin{corollary} \label{corollary1} Let $0\to \mathfrak{r}\to \mathfrak{f}\to \mathfrak{g} \to 0$ be a free presentation of a Leibniz algebra
$\mathfrak{g}$. Then there is an isomorphism
\[
\mathfrak{g}\curlywedge \mathfrak{g} \cong [\mathfrak{f}, \mathfrak{f}]/[\mathfrak{r}, \mathfrak{f}].
\]
\end{corollary}
\begin{proof} This follows from Lemma \ref{lemma1} and Proposition \ref{proposition1}.
\end{proof}

\begin{theorem} \label{theorem1} Let $\mathfrak{g}$ be a Leibniz algebra. Then there is an isomorphism
\[
HL_2(\mathfrak{g}) \cong \Ker \Big(\theta_{\mathfrak{g}, \mathfrak{g}}\colon \mathfrak{g}\curlywedge \mathfrak{g}\to \mathfrak{g}\Big).
\]
\end{theorem}
\begin{proof} Let $0\to \mathfrak{r}\to \mathfrak{f}\to \mathfrak{g} \to 0$ be a free presentation of $\mathfrak{g}$. By the Hopf
formula we have
\[
HL_2(\mathfrak{g}) \cong \Ker \Big([\mathfrak{f}, \mathfrak{f}]/[\mathfrak{r}, \mathfrak{f}]\to \mathfrak{g}\Big).
\]
Thus, Corollary \ref{corollary1} completes the proof.
\end{proof}

\begin{proposition} \label{proposition2} Let $0\to \mathfrak{a}\to \mathfrak{g}\to \mathfrak{h} \to 0$ be an extension of Leibniz
algebras. Then we have the following exact sequence
\begin{align*}
&\Ker \big( \theta_\mathfrak{a, g} \colon \mathfrak{a}\curlywedge \mathfrak{g} \to \mathfrak{a}\big) \to HL_2(\mathfrak{g})\to HL_2(\mathfrak{h})\to
\mathfrak{a}/[\mathfrak{a}, \mathfrak{g}] \to HL_1(\mathfrak{g})
\to HL_1(\mathfrak{h})\to 0 .
\end{align*}
\end{proposition}

\begin{proof} By Lemma \ref{lemma1} we have the following commutative diagram with exact rows
\begin{equation*}
\xymatrix@+20pt{
&\mathfrak{a}\curlywedge \mathfrak{g}\ \ar@{->}[r]
\ar@{->}[d]_{\theta_\mathfrak{a, g}}
 &\mathfrak{g}\curlywedge \mathfrak{g}\ar@{->}[r]
\ar@{->}[d]_{\theta_\mathfrak{g, g}}
&\mathfrak{h}\curlywedge \mathfrak{h}\ar@{->}[r]
\ar@{->}[d]_{\theta_\mathfrak{h, h}}
&0 \\
0\ \ar@{->}[r]
&\mathfrak{a}\ \ar@{->}[r]
 &\mathfrak{g}\ar@{->}[r]
&\mathfrak{h}\ar@{->}[r]
&0 .
}\end{equation*}
Now the Snake Lemma and Theorem \ref{theorem1} yield the exact sequence.
\end{proof}

Let $\mathfrak{g} \bullet \mathfrak{g}$ denote the vector space
$\Coker (\mathfrak{g\otimes }\mathfrak{g}\otimes \mathfrak{g} \overset{d}{\to } \mathfrak{g}\otimes \mathfrak{g})$, where
$d$ is the boundary map in $CL_*(\mathfrak{g})$. Let $\de \colon \mathfrak{g} \bullet \mathfrak{g} \to \mathfrak{g} \curlywedge \mathfrak{g}$
be a linear map given by $x \bullet y \mapsto x \curlywedge y$, where $x\bullet y$ denotes the coset of $x\otimes y \in
\mathfrak{g}\otimes \mathfrak{g}$ into $\mathfrak{g}\bullet \mathfrak{g}$. It is easy to check that $\de$ is well-defined.

\begin{proposition}\label{} The linear map $\de \colon \mathfrak{g} \bullet \mathfrak{g} \to \mathfrak{g} \curlywedge \mathfrak{g}$ is
an isomorphism of vector spaces.
\end{proposition}
\begin{proof} We have the following commutative diagram with exact rows:

\begin{equation*}
\xymatrix@+20pt{
0\ \ar@{->}[r] &\Ker d'\ \ar@{->}[r]
\ar@{->}[d]_{}
 &\mathfrak{g}\bullet \mathfrak{g}\ar@{->}[r]_{d'}
\ar@{->}[d]_{\de}
& \mathfrak{g}\ar@{->}[r]
\ar@{=}[d]
&\mathfrak{g}/[\mathfrak{g}, \mathfrak{g}] \ar@{=}[d] \\
0\ \ar@{->}[r]
&\Ker \theta_{\mathfrak{g}, \mathfrak{g}}\ \ar@{->}[r]
 &\mathfrak{g}\curlywedge \mathfrak{g}\ar@{->}[r]_{\theta_\mathfrak{g, g}}
&\mathfrak{g}\ar@{->}[r]
&\mathfrak{g}/[\mathfrak{g}, \mathfrak{g}] ,
}\end{equation*}

where $d'$ is given by $x\bullet y \mapsto [x, y]$ for each $x, y \in \mathfrak{g}$. Since $HL_2(\mathfrak{g})=\Ker d'$,
by Theorem \ref{theorem1} we have an isomorphism $\Ker d' \cong \Ker \theta_{\mathfrak{g, g}}$. It is easy to verify that this isomorphism
is induced by $\de$. Hence, the above diagram proves the proposition.
\end{proof}

\begin{remark}  It is shown in  \cite{KuPi} that the vector space $\mathfrak{g}\bullet \mathfrak{g}$ has the Leibniz algebra structure
given by
$$
[x\bullet y , x' \bullet y'] = [x, y]\bullet [x', y'],
$$
for each $x, y, x', y' \in \mathfrak{g}$. This fact results from the previous proposition, because in $\mathfrak{g}\curlywedge \mathfrak{g}$ we have that $[x\curlywedge y , x' \curlywedge y'] = [x, y]\curlywedge [x', y']$.
\end{remark}

\

\section{Third homology and the eight term exact sequence}\label{S:H3}

In this section we will use the method developed in \cite{DIL16} to prove the eight term exact sequence.

\begin{lemma} \label{lemma2} Let $0\to \mathfrak{a}\to \mathfrak{g}\overset{\tau}{\to} \mathfrak{h} \to 0$ be a split extension of Leibniz
algebras, i. e. there is a homomorphism of Leibniz algebras $\sigma \colon \mathfrak{h}\to \mathfrak{g}$ such that $\tau\circ \sigma =1_\mathfrak{h}$.
Then the induced homomorphism of Leibniz algebras $\mathfrak{a}\curlywedge \mathfrak{g}\to \mathfrak{g}\curlywedge \mathfrak{g}$ is injective.
\end{lemma}
\begin{proof}  Denote the induced homomorphism of Leibniz algebras {$\mathfrak{a}\curlywedge \mathfrak{g}\to \mathfrak{g}\curlywedge \mathfrak{g}$} by $\al$.
We shall show that there exists a linear map $\mathfrak{g}\curlywedge \mathfrak{g} \to \mathfrak{a}\curlywedge \mathfrak{g}$
which is a $\mathbb{K}$-linear splitting for $\alpha$.
For each element $x\in \mathfrak{g}$ there are unique $a\in \mathfrak{a}$ and $y\in \mathfrak{h}$ such that $x=a+\tau (h)$.
Let $\be \colon \mathfrak{g}\otimes \mathfrak{g} \to \mathfrak{a}\curlywedge \mathfrak{g}$ be a linear map given by
$(a+\tau (h)) \otimes (a'+\tau (h')) \overset{\be}{\mapsto} a\curlywedge \tau (h') + a \curlywedge a' +
\tau (h) \curlywedge a'$ for each $a, a'\in \mathfrak{a}$ and $y, y' \in \mathfrak{g}$.
It is easy to check that $\be$ is well-defined $\mathbb{K}$-linear map and that $\be (\Img d) = 0$, where
$d \colon \mathfrak{g\otimes }\mathfrak{g}\otimes \mathfrak{g} \to  \mathfrak{g}\otimes \mathfrak{g}$ is the boundary map
in $CL_*(\mathfrak{g})$. Thus, $\be$ induces the linear map $\bar{\be} \colon \mathfrak{g}\bullet \mathfrak{g}\to
\mathfrak{a}\curlywedge \mathfrak{g}$. Now, let $\de \colon \mathfrak{g} \bullet \mathfrak{g} \to \mathfrak{g} \curlywedge \mathfrak{g}$
be the linear map defined as in Section \ref{S:Ten}. Then,
the linear map $\bar{\be}\de^{-1} \colon \mathfrak{g} \curlywedge \mathfrak{g} \to \mathfrak{a} \curlywedge \mathfrak{g}$
is such that $\bar{\be}\de^{-1}\alpha =1_{\mathfrak{a}\curlywedge \mathfrak{g}}$. Thus, $\al$ is injective.
\end{proof}

\begin{theorem} \label{theorem2} Let $0\to\mathfrak{r}\to \mathfrak{f}\to \mathfrak{g}\to 0$ be a free presentation
of a Leibniz algebra $\mathfrak{g}$. Then there is an isomorphism
$$
HL_3(\mathfrak{g}) \cong \Ker \Big(\theta_{\mathfrak{r}, \mathfrak{f}}\colon \mathfrak{r}\curlywedge \mathfrak{f}\to \mathfrak{r}\Big).
$$
\end{theorem}
\begin{proof} According to Remark \ref{remark1}, for computing $HL_*(\mathfrak{g})$ we can use an exact augmented simplicial Leibniz algebra
\[
 \cdots\; \mathfrak{f}_2\;
\athree{d^2_0}{d^2_2} \; \mathfrak{f}_1\atwo{d^1_0}{d^1_1} \; \mathfrak{f}_0 \overset{d^0_0}{\to} \mathfrak{g}
\]
such that $\mathfrak{f}_i$ is a free Leibniz algebra over a set, for each $i\geq 0$, $\mathfrak{f}_0=\mathfrak{f}$ and $\Ker d^0_0 = \mathfrak{r}$. Then, the long exact homotopy sequence derived from the following short exact sequence of simplicial Leibniz algebras
\[
0\to [\mathfrak{f}_*,\mathfrak{f}_*]\to \mathfrak{f}_*\to \mathfrak{Ab}(\mathfrak{f}_*)\to 0,
\]
implies that $HL_3(\mathfrak{g})$ is isomorphic to the first homotopy group of the following simplicial Leibniz algebra
\[
\cdots\; [\mathfrak{f}_2, \mathfrak{f}_2]\;
\athree{d^2_0}{d^2_2} \; [\mathfrak{f}_1, \mathfrak{f}_1]\atwo{d^1_0}{d^1_1} \; [\mathfrak{f}_0, \mathfrak{f}_0].
\]
Hence,
\[
HL_3(\mathfrak{g})\cong \Ker d^1_0 \cap \Ker d^1_1 \cap [\mathfrak{f}_1, \mathfrak{f}_1] /
d^2_2\big ( \Ker d^2_0 \cap \Ker d^2_1 \cap [\mathfrak{f}_2, \mathfrak{f}_2] \big).
\]
Since $HL_2(\mathfrak{f}_0)=0$ and $HL_3(\mathfrak{f}_1)=0$, using Hopf formulas we have
\begin{align*}
&\Ker d^1_0 \cap [\mathfrak{f}_1, \mathfrak{f}_1]= [\Ker d^1_0, \mathfrak{f}_1],\\
&\Ker d^2_0 \cap \Ker d^2_1 \cap [\mathfrak{f}_2, \mathfrak{f}_2]=[\Ker d^2_0 \cap \Ker d^2_1, \mathfrak{f}_2]+
[\Ker d^2_0, \Ker d^2_1].
\end{align*}
Therefore,
\begin{align*}
HL_3(\mathfrak{g})&\cong \Ker d^1_1 \cap [\Ker d^1_0, \mathfrak{f}_1] /
d^2_2\big( [\Ker d^2_0 \cap \Ker d^2_1, \mathfrak{f}_2]+[\Ker d^2_0, \Ker d^2_1] \big) \\
&=
\Ker d^1_1 \cap [\Ker d^1_0, \mathfrak{f}_1] /
\big( [d^2_2(\Ker d^2_0 \cap \Ker d^2_1), d^2_2(\mathfrak{f}_2)]+[d^2_2(\Ker d^2_0), d^2_2(\Ker d^2_1)] \big)\\
&=
\Ker d^1_1 \cap [\Ker d^1_0, \mathfrak{f}_1] /
\big( [\Ker d^1_0 \cap \Ker d^1_1, \mathfrak{f}_1]+[\Ker d^1_0, \Ker d^1_1] \big).
\end{align*}
Since $d_1^1 \big ( [\Ker d^1_0 \cap \Ker d^1_1, \mathfrak{f}_1]+[\Ker d^1_0, \Ker d^1_1] \big)=0$, we get
\begin{equation}\label{E:HS3}
HL_3(\mathfrak{g})\cong \Ker \Bigg( \frac{[\Ker d^1_0, \mathfrak{f}_1] }{
[\Ker d^1_0 \cap \Ker d^1_1, \mathfrak{f}_1]+[\Ker d^1_0, \Ker d^1_1] } \overset{d_1^1}{\longrightarrow}[\mathfrak{f}_0, \mathfrak{f}_0]\Bigg).
\end{equation}
Furthermore, since $0\to \Ker d^1_1 \to \mathfrak{f}_1 \overset{d^1_1}{\to} \mathfrak{f}_0\to 0$ is a free presentation of $\mathfrak{f}_0$ which splits,
 by Proposition \ref{proposition1} and Lemma \ref{lemma2} the following map $\Ker d^1_1 \curlywedge \mathfrak{f}_1 \to [\Ker d^1_1, \mathfrak{f}_1]$,
  defined by $x\curlywedge y \mapsto [x,y]$, $y\curlywedge x \mapsto [y,x]$,
for all $x\in \Ker d_1^1$ and $y\in \mathfrak{f}_1$, is an isomorphism. Therefore,
\begin{equation*}
\frac{[\Ker d^1_0, \mathfrak{f}_1] }{[\Ker d^1_0 \cap \Ker d^1_1, \mathfrak{f}_1]+[\Ker d^1_0, \Ker d^1_1] }\cong
\frac{\Ker d^1_0\curlywedge \mathfrak{f}_1 }{\big((\Ker d^1_0 \cap \Ker d^1_1) \curlywedge \mathfrak{f}_1\big)+\big(\Ker d^1_0\curlywedge \Ker d^1_1\big)}\,.
\end{equation*}
Hence, the  exact sequences $0\to \Ker d^1_0 \cap \Ker d^1_1 \to \Ker d^1_0 \to \Ker d^0_0 \to 0$ and
$0\to \Ker d^1_1 \to \mathfrak{f}_1 \to \mathfrak{f}_0 \to 0$, and (\ref{sequence1}) imply that
\begin{equation}\label{E:rf}
\frac{[\Ker d^1_0, \mathfrak{f}_1] }{[\Ker d^1_0 \cap \Ker d^1_1, \mathfrak{f}_1]+[\Ker d^1_0, \Ker d^1_1] }\cong
\Ker d_0^0 \curlywedge \mathfrak{f}_0 = \mathfrak{r}\curlywedge \mathfrak{f}.
\end{equation}
Now \eqref{E:HS3} and \eqref{E:rf} complete the proof.
\end{proof}

\begin{proposition} \label{proposition3} Let $0\to \mathfrak{a}\to \mathfrak{g}\overset{\tau}{\to} \mathfrak{h} \to 0$ be an extension of Leibniz
algebras. Then we have the following exact sequence
\begin{align*}
HL_3(\mathfrak{g})&\to HL_3(\mathfrak{h})\to \Ker \big( \theta_\mathfrak{a, g} \colon \mathfrak{a}\curlywedge \mathfrak{g} \to \mathfrak{a}\big) \to HL_2(\mathfrak{g})\to HL_2(\mathfrak{h})\\
&\to
\mathfrak{a}/[\mathfrak{a}, \mathfrak{g}] \to HL_1(\mathfrak{g})\to HL_1(\mathfrak{h})\to 0 .
\end{align*}
\end{proposition}
\begin{proof} Any free presentation $0 \to \mathfrak{r} \to \mathfrak{f} \overset{\rho}{\to} \mathfrak{g} \to 0$ of $\mathfrak{g}$ produces a free presentation $0\to \mathfrak{s}\to \mathfrak{f}\overset{\tau\circ\rho}{\to} \mathfrak{h} \to 0$ of $\mathfrak{h}$ and
an extension $0 \to \mathfrak{r} \to \mathfrak{s} \to \mathfrak{a} \to 0$ of Leibniz algebras.
By (\ref{sequence1}) we have the following exact sequence
\[
\mathfrak{s}\curlywedge \mathfrak{r} \times \mathfrak{r}\curlywedge \mathfrak{f} \to \mathfrak{s}\curlywedge \mathfrak{f} \to
\mathfrak{a}\curlywedge \mathfrak{g} \to 0.
\]
This sequence yields the following exact sequence
\[
\mathfrak{r}\curlywedge \mathfrak{f} \to \mathfrak{s}\curlywedge \mathfrak{f} \to \mathfrak{a}\curlywedge \mathfrak{g} \to 0.
\]
Thus, we have the following commutative diagram with exact rows
\begin{equation*}
\xymatrix@+20pt{
&\mathfrak{r}\curlywedge \mathfrak{f}\ \ar@{->}[r]
\ar@{->}[d]_{\theta_\mathfrak{r, f}}
 &\mathfrak{s}\curlywedge \mathfrak{f}\ar@{->}[r]
\ar@{->}[d]_{\theta_\mathfrak{s, f}}
&\mathfrak{a}\curlywedge \mathfrak{g}\ar@{->}[r]
\ar@{->}[d]_{\theta_\mathfrak{a, g}}
&0 \\
0\ \ar@{->}[r]
&\mathfrak{r}\ \ar@{->}[r]
 &\mathfrak{s}\ar@{->}[r]
&\mathfrak{a}\ar@{->}[r]
&0 .
}\end{equation*}
The Snake Lemma and Theorem \ref{theorem2} imply the following exact sequence
\[
HL_3(\mathfrak{g})\to HL_3(\mathfrak{h})\to \Ker \big( \theta_\mathfrak{a, g} \colon \mathfrak{a}\curlywedge \mathfrak{g}
\to \mathfrak{a}\big) \to \mathfrak{r}/[\mathfrak{r}, \mathfrak{f}]\to \mathfrak{s}/[\mathfrak{s}, \mathfrak{f}].
\]
It is easy to see that
\[
\Img \big (\Ker \big( \theta_\mathfrak{a, g} \colon \mathfrak{a}\curlywedge \mathfrak{g}
\to \mathfrak{a}\big) \to \mathfrak{r}/[\mathfrak{r}, \mathfrak{f}]\big) \subseteq
\frac{\mathfrak{r}\cap [\mathfrak{f}, \mathfrak{f}]}{[\mathfrak{r}, \mathfrak{f}]}.
\]
Therefore, we have an exact sequence:
\[
HL_3(\mathfrak{g})\to HL_3(\mathfrak{h})\to \Ker \big( \theta_\mathfrak{a, g} \colon \mathfrak{a}\curlywedge \mathfrak{g}
\to \mathfrak{a}\big) \to \frac{\mathfrak{r}\cap [\mathfrak{f}, \mathfrak{f}]}{[\mathfrak{r}, \mathfrak{f}]}
\to \frac{\mathfrak{s}\cap [\mathfrak{f}, \mathfrak{f}]}{[\mathfrak{s}, \mathfrak{f}]}.
\]
Using the Hopf formula we get an exact sequence:
\[
HL_3(\mathfrak{g})\to HL_3(\mathfrak{h})\to \Ker \big( \theta_\mathfrak{a, g} \colon \mathfrak{a}\curlywedge \mathfrak{g}
\to \mathfrak{a}\big) \to HL_2(\mathfrak{g})\to HL_2(\mathfrak{h}).
\]
The rest of the proof follows from Proposition \ref{proposition2}.
\end{proof}

\begin{corollary} (see \cite{CaPi00}) Let $0\to \mathfrak{a}\to \mathfrak{g}\to \mathfrak{h} \to 0$ be a central extension of Leibniz
algebras, i.e. $[a, x]=[x, a]=0$ for all $a\in \mathfrak{a}$ and $x\in \mathfrak{g}$. Then there is the following exact sequence
\[
HL_3(\mathfrak{g})\to HL_3(\mathfrak{h})\to \Coker \left( \mathfrak{a}\otimes \mathfrak{a} \overset{\eta}{\to}
\mathfrak{a}\otimes \frac{\mathfrak{g}}{[\mathfrak{g}, \mathfrak{g}]} \oplus
\frac{\mathfrak{g}}{[\mathfrak{g}, \mathfrak{g}]}\otimes \mathfrak{a}\right) \to HL_2(\mathfrak{g})\to HL_2(\mathfrak{h})\to 0,
\]
where $\eta \colon \mathfrak{a}\otimes \mathfrak{a} \to \mathfrak{a}\otimes \frac{\mathfrak{g}}{[\mathfrak{g}, \mathfrak{g}]}\oplus
\frac{\mathfrak{g}}{[\mathfrak{g}, \mathfrak{g}]}\otimes \mathfrak{a}$ is given by $a\otimes b\mapsto (a\otimes \overline{b}, -\overline{a}\otimes b)$,
where $\overline{a}=a+[\mathfrak{g}, \mathfrak{g}]$ and $\overline{b}=b+[\mathfrak{g}, \mathfrak{g}]$ for each $a, b\in \mathfrak{a}$.
\end{corollary}
\begin{proof}
Under the required conditions, the Leibniz algebras  $\mathfrak{a}$ and $\mathfrak{g}$ act trivially on each other. Then,  by \cite[Proposition 4.2]{Gn99},  we have a natural isomorphism
\[
\mathfrak{a}\star \mathfrak{g}\cong \mathfrak{a}\otimes \frac{\mathfrak{g}}{[\mathfrak{g}, \mathfrak{g}]}\oplus
\frac{\mathfrak{g}}{[\mathfrak{g}, \mathfrak{g}]}\otimes \mathfrak{a} .
\]
Since $\mathfrak{a}\curlywedge \mathfrak{g}$ is obtained from $\mathfrak{a}\star \mathfrak{g}$ by killing the elements of the form $a*i(b)-i(a)*b$, where $a, b\in \mathfrak{a}$ and $i\colon \mathfrak{a}\to \mathfrak{g}$ is the natural inclusion, we get an isomorphism
\[
\mathfrak{a}\curlywedge \mathfrak{g}\cong\Coker \Big( \eta \colon \mathfrak{a}\otimes \mathfrak{a} \to
\mathfrak{a}\otimes \frac{\mathfrak{g}}{[\mathfrak{g}, \mathfrak{g}]}\oplus
\frac{\mathfrak{g}}{[\mathfrak{g}, \mathfrak{g}]}\otimes \mathfrak{a}\Big).
\]
Then, by Proposition \ref{proposition3} we conclude the required result.
\end{proof}

\

\section{Relationship to the universal quadratic functor}\label{S:quad}

%\subsection{Universal quadratic functor}

In this section $\mathbb{K}$ is a commutative ring with identity (not necessarily a field). Keeping in mind Remark \ref{R: ring},  we will use only those constructions and facts from the previous sections, which do not require $\mathbb{K}$ to be a field.

In the case of Lie algebras, there is a connection between the non-abelian exterior product of Lie algebras and Whitehead's universal quadratic functor (\cite{Ell91}). We can observe it in the Leibniz algebras case too.

\begin{definition}[\cite{SiTy}]
	Let $A$ be a $\mathbb{K}$-module and consider the endofunctor that sends $A$ to the $\mathbb{K}$-module generated by the symbols $\gamma(a)$ with $a \in A$, quotient by the submodule generated by
	\begin{align*}
	k^2\gamma(a) &= \gamma(ka), \\
	\gamma(a + b + c) + \gamma(a) + \gamma(b) + \gamma(c) &= \gamma(a + b) + \gamma(a + c) + \gamma(b + c), \\
	\gamma(ka + b)  + k\gamma(a) + k\gamma(b) &= k\gamma(a + b) + \gamma(ka) + \gamma(b),
	\end{align*}
	for all $k \in \mathbb{K}$ and $a, b, c \in A$.
	This functor denoted by $\Gamma$ is called \emph{universal quadratic functor}.
\end{definition}

\begin{proposition}[\cite{SiTy}]\label{P:basis}
	Let $I$ be a well-ordered set and $A$ be a free $\mathbb{K}$-module with basis $\{e_i\}_{i \in I}$. Then $\Gamma(A)$ is a free $\mathbb{K}$-module with basis
	\[
	\{\gamma(e_i)\}_{i \in I} \cup \{ \gamma(e_i + e_j) - \gamma(e_i) - \gamma(e_j) \}_{i < j}
	\]	
\end{proposition}

Let $\eta \colon \mm \to \Gg$ and $\mu \colon \nn \to \Gg$ be two crossed modules of Leibniz algebras. As we know there are induced actions of $\mm$ and $\nn$  on each other via the action of $\Gg$.
Let $\mm \times_{\Gg} \nn = \{ (m, n) \mid \eta(m) = \mu(n) \}$ be the pullback of $\eta$ and $\mu$. It is a Leibniz subalgebra of $\mm \oplus \nn$. Let $\langle \mm, \nn \rangle = \{(\tau_{\mm}(x), \tau_{\nn}(x))  \mid x \in \mm \star \nn \}$, where $\tau_{\mm}\colon \mm \star \nn\to \mm $ and $\tau_{\nn}\colon \mm \star \nn\to \nn$ are homomorphisms introduced in Subsection \ref{S:tensor}.

\begin{proposition}\label{P:ideal}
$\langle \mm, \nn \rangle$ is an ideal of $\mm \times_{\Gg} \nn$  and the quotient $(\mm \times_{\Gg} \nn) / \langle \mm, \nn \rangle$ is abelian.
\end{proposition}

\begin{proof}
	The assertion that $\langle \mm, \nn \rangle$ is an ideal of $\mm \times_{\Gg} \nn$ follows by straightforward calculations. For instance, given any $m\in \mm $, $n\in \nn$ and $(m', n')\in \mm \times_{\Gg} \nn$ we get
	\begin{align*}
		[(\tau_{\mm}(m*n), \tau_{\nn}(m*n)), (m', n')] &= [(m^n, \;^mn), (m', n')] \\
		{} &= ([m^n, m'], [^mn, n']) \\
		{} &= ((m^n)^{m'}, \;^{\mu(^mn)}n') \\
		{} &= ( (m^n)^{n'}, \;^{(m^n)}n' ) \\
		{} &= (\tau_{\mm}(m^n * n'), \tau_{\nn}(m^n * n')).
	\end{align*}
%\textcolor{red}{and similar ones for the other elements, we deduce that $\langle \mm, \nn \rangle$ is an ideal of $\mm \times_{\Gg} \nn$ }.
	
	Now take any $(m, n)$, $(m', n') \in \mm \times_{\Gg} \nn$, then we have
	\[
	[(m, n), (m', n')] = ([m, m'], [n, n']) = (m^{m'}, \; ^nn') = (m^{n'}, \;^mn') =  (\tau_{\mm}(m* n'), \tau_{\nn}(m* n')),
	\]
	showing that $(\mm \times_{\Gg} \nn) / \langle \mm, \nn \rangle$ is abelian.
\end{proof}

\begin{proposition}\label{P:complex}
	There is a well defined homomorphism of Leibniz algebras
	\[\xymatrix{
		\Gamma\bigg(\dfrac{\mm \times_{\Gg} \nn}{\langle \mm, \nn \rangle}\bigg) \ar[r]^-{\psi} &  \mm \star  \nn ,
	}
	\]
	 given by  $\psi(\gamma ((m, n)+\langle \mm, \nn \rangle )) = m * n - n*m$.
\end{proposition}

\begin{proof}
It is easy to check that  $\psi$ preserves the defining relations of $\Gamma$. Thus, it suffices to show that $ ( m'+\tau_{\mm}(x)) *( n'+\tau_{\nn}(x))- ( n'+\tau_{\nn}(x)) *(m'+ \tau_{\mm}(x))=    m' * n' - n'*m'$
for each $x\in \mm \star \nn$. This  reduces to prove that $m'*\;^mn - \;^mn * m' + m^n*n' - n'*m^n = 0$. Using the defining relations of the non-abelian tensor product we have
\begin{align*}
	m'*\;^mn - \;^mn*m' + &m^n*n' - n'*m^n  \\
	{}=& [m', m]*n - \;^{m'}n*m - m*n^{m'} - [m, m']*n \\
	{}& + m^{n'}*n + m*[n, n'] -\;^{n'}m*n + [n', n]*m  \\
	{}=& [m', m]*n - [n', n] * m - m*[n, n'] - [m, m']* n \\
	{}& + [m, m'] * n + m*[n, n'] + [m', m]*n - [n', n] * m = 0.
\end{align*}
 By Proposition \ref{P:cen} we know that $\Img \psi$ is contained in the centre of $\mm \star \nn$, so $\psi$ is a homomorphism of Leibniz algebras.
\end{proof}

It is clear that $\Img \psi$ is contained in $\Ker (\pi\colon \mathfrak{m}\star \mathfrak{n} \to \mathfrak{m}\curlywedge \mathfrak{n})$, where $\pi$ is the canonical projection.
But the following sequence
\[\xymatrix{
		\Gamma\bigg( \dfrac{\mm \times_{\Gg} \nn }{\langle \mm, \nn \rangle}  \bigg)
 \ar[r]^-{{\psi}} & \mm \star \nn \ar[r]^-{\pi} & \mm \curlywedge \nn \ar[r] & 0,
	}
	\]
is not exact in many cases. Nevertheless, we have the following

\begin{proposition}\label{P:complex1}
	There is an exact sequence of Leibniz algebras
	\[\xymatrix{
		\Gamma\bigg( \dfrac{\mm \times_{\Gg} \nn }{\langle \mm, \nn \rangle}  \oplus   \dfrac{\mm \times_{\Gg} \nn }{\langle \mm, \nn \rangle}  \bigg)
 \ar[r]^-{\widetilde{\psi}} & \mm \star \nn \ar[r]^-{\pi} & \mm \curlywedge \nn \ar[r] & 0,
	}
	\]
	where $\widetilde{\psi}( \gamma ((m, n)+\langle \mm, \nn \rangle , (m', n')+\langle \mm, \nn \rangle)) = m * n' - n*m'$.
\end{proposition}

\begin{proof}
Like in Proposition \ref{P:complex}, the crucial part of the proof is to show that
$ (m + \tau_{\mm}(x)) * (n'+\tau_{\nn}(x')) -( n+\tau_{\nn}(x))*(m'+\tau_{\mm}(x')) = m * n' - n*m'$ for all  $x, x'\in \mm \star \nn$.
Let $x = m_1 * n_1$ and $x' = m_1' * n_1'$, then proving that $m*\; ^{m_1'}n_1' + m_1^{n_1} * n' - n * m_1'^{n_1'} - \; ^{m_1}n_1 * m'=0$, will imply the result.
Using the defining identities of the non-abelian tensor product, we get
\begin{align*}
m*\; ^{m_1'}n_1' + m_1^{n_1} * n' &- n * m_1'^{n_1'} - \; ^{m_1}n_1 * m' \\
{} =& [m, m_1'] * n_1' - \;^mn_1'*m_1' + m_1*[n_1, n'] + m_1^{n'}*n_1 \\
{}&+[n, n_1'] * m_1' - \;^nm_1' * n_1' - [m_1, m'] * n_1 - m_1*n_1^{m'} = 0.
\end{align*}
The proof is complete, since it is straightforward that in this case $  \Img\widetilde{\psi} = \Ker \pi$.
\end{proof}

In the particular case of $\Aa$ and $\bb$ being ideals of a Leibniz algebra $\Gg$, Proposition \ref{P:complex} and Proposition \ref{P:complex1} can be viewed as follows:

\begin{corollary}
	There is a well defined homomorphism of Leibniz algebras
	\[\xymatrix{
		\Gamma\bigg(\dfrac{\mathfrak{a} \cap \mathfrak{b}}{[\mathfrak{a}, \mathfrak{b}]}\bigg) \ar[r]^-{\psi} & \mathfrak{a} \star \mathfrak{b} ,
	}
	\]
	given by $\psi(\gamma (c+[\mathfrak{a}, \mathfrak{b}])) = i_1(c) * i_2(c) - i_2(c) * i_1(c)$, for any $c \in \mathfrak{a}\cap \mathfrak{b}$ and the natural inclusions $i_1\colon \mathfrak{a} \cap \mathfrak{b}\to \mathfrak{a}$, $i_2\colon \mathfrak{a} \cap \mathfrak{b}\to \mathfrak{b}$.
\end{corollary}

\begin{corollary}
	There is an exact sequence of Leibniz algebras
	\[\xymatrix{
		\Gamma\bigg(\dfrac{\mathfrak{a} \cap \mathfrak{b}}{[\mathfrak{a}, \mathfrak{b}]} \oplus \dfrac{\mathfrak{a} \cap \mathfrak{b}}{[\mathfrak{a}, \mathfrak{b}]}  \bigg)
\ar[r]^-{\widetilde{\psi}} & \mathfrak{a} \star \mathfrak{b} \ar[r]^-{\pi} & \mathfrak{a} \curlywedge \mathfrak{b} \ar[r] & 0,
	}
	\]
	where $\widetilde{\psi}(c+[\mathfrak{a}, \mathfrak{b}], c'+[\mathfrak{a}, \mathfrak{b}]) = i_1(c) * i_2(c') - i_2(c) * i_1(c')$, for all $c, c' \in \mathfrak{a}\cap \mathfrak{b}$.
\end{corollary}

In the next proposition $\dfrac{\mathfrak{a} \cap \mathfrak{b}}{[\mathfrak{a}, \mathfrak{b}]} \wedge \dfrac{\mathfrak{a} \cap \mathfrak{b}}{[\mathfrak{a}, \mathfrak{b}]} $
denotes the exterior product of $\mathbb{K}$-modules.

\begin{proposition}\label{P: new}
	There is an exact sequence of Leibniz algebras
	\[\xymatrix{
		\dfrac{\mathfrak{a} \cap \mathfrak{b}}{[\mathfrak{a}, \mathfrak{b}]} \wedge \dfrac{\mathfrak{a} \cap \mathfrak{b}}{[\mathfrak{a}, \mathfrak{b}]}
\ar[r]^-{\phi} & \dfrac{\mathfrak{a} \star \mathfrak{b}}{\Img \psi} \ar[r]^-{\overline{\pi}} & \mathfrak{a} \curlywedge \mathfrak{b} \ar[r] & 0,
	}
	\]
	where $\overline{\pi}$ is the canonical projection and
$\phi((c+[\mathfrak{a}, \mathfrak{b}])\wedge ( c'+[\mathfrak{a}, \mathfrak{b}])) = i_1(c) * i_2(c') - i_2(c) * i_1(c')+\Img \psi$, for all $c, c' \in \mathfrak{a}\cap \mathfrak{b}$.
\end{proposition}

\begin{proof} To check that $\phi$ is well defined, it suffices to verify the following identities:
\[
	 i_1(c) * i_2[a, b] - i_2(c) * i_1[a, b] =0 ,
\]
\[
	 i_1[a, b] * i_2(c) - i_2[a, b] * i_1(c) =0 ,
\]
for all $c\in \mathfrak{a} \cap \mathfrak{b}$, $a\in \mathfrak{a}$ and $b\in \mathfrak{b}$. Using ($3\textrm{c}$), ($3\textrm{d}$) and ($4\textrm{b}$) we have:

\begin{align*}
	 i_1&(c) * i_2[a, b] - i_2(c) * i_1[a, b] \\
& = 	  i_2[c, a] * i_1( b) -  i_1[c, b] * i_2( a)  - i_2(c) * i_1[a, b] \\
 & = 	 i_1[c, b] * i_2( a) - i_2(c) * i_1[b, a]-  i_1[c, b] * i_2( a)  - i_2(c) * i_1[a, b] =0.
\end{align*}

Using  ($3\textrm{d}$), ($3\textrm{b}$) and ($4\textrm{b}$) we have:

\begin{align*}
	  i_1&[a, b] * i_2(c) - i_2[a, b] * i_1(c)  \\
&= 	   i_2[a, c]* i_1(b)+  i_2(a)* i_1[b, c]-  i_2[a, b]* i_1(c)\\
&=    i_2(a)* i_1[c, b]+  i_2(a)* i_1[b, c] = 0.
\end{align*}

The proof is complete, since it is straightforward that $ \Img\phi = \Ker \overline{\pi}$.
\end{proof}

Let $\mathfrak{g}$ be a Leibniz algebra, and let $\tau \colon \mathfrak{g} \star \mathfrak{g} \to  \mathfrak{g}^{\ab} \otimes \mathfrak{g}^{\ab}$  be the homomorphism defined by
$i_1(g)*i_2(g' )\mapsto (g+ [\mathfrak{g},  \mathfrak{g}])\otimes (g' +  [\mathfrak{g},  \mathfrak{g}])$, $i_2(g)*i_1(g' )\mapsto 0$,
for all $g, g'\in \mathfrak{g}$, where $\otimes$ denotes the tensor product of $\mathbb{K}$-modules. Then, $\tau$ induces well defined homomorphisms
$\overline {\tau}\colon \mathfrak{g} \star \mathfrak{g} \to  \mathfrak{g}^{\ab} \wedge \mathfrak{g}^{\ab}$ and
$\widetilde{\tau}\colon  \dfrac{\mathfrak{g} \star \mathfrak{g}}{\Img \psi} \to  \mathfrak{g}^{\ab} \wedge \mathfrak{g}^{\ab}$, where $\psi$ is defined as in Proposition \ref{P:complex}.

\begin{proposition}\label{P: new1}
	We have the following exact sequence of Leibniz algebras
	
\[\xymatrix{
			0 \ar[r] & \mathfrak{g}^{\ab} \wedge  \mathfrak{g}^{\ab}  \ar[r]^-{\phi} &
\dfrac{\mathfrak{g} \star \mathfrak{g}}{\Img \psi} \ar[r]^-{\overline{\pi}} & \mathfrak{g} \curlywedge \mathfrak{g} \ar[r] & 0.
		}
		\]
where $\overline{\pi}$ is the canonical projection and $\phi$ is defined as in Proposition \ref{P: new}. Moreover, the following map
\[\xymatrix{
\dfrac{\mathfrak{g} \star \mathfrak{g}}{\Img \psi} \ar[r]^-{(\overline{\pi}, \widetilde{\tau})} & (\mathfrak{g} \curlywedge \mathfrak{g}) \oplus  (\mathfrak{g}^{\ab} \wedge \mathfrak{g}^{\ab})
		}
		\]
is an isomorphism of Leibniz algebras.
\end{proposition}

\begin{proof} It is easy to see that $\widetilde{\tau} \circ \phi =1_{\mathfrak{g}^{\ab} \wedge  \mathfrak{g}^{\ab}}$. This implies both parts of the proposition.
\end{proof}

\begin{proposition} Let $\mathfrak{g}$ be a Leibniz algebra such that $\mathfrak{g}^{\ab}$ is free as a $\mathbb{K}$-module. Then there is an exact sequence of Leibniz algebras
	\[\xymatrix{
			0 \ar[r] & \Gamma(\mathfrak{g}^{\ab}) \ar[r]^-{\psi} & \mathfrak{g} \star \mathfrak{g} \ar[r]^-{(\pi, \overline{\tau})} & (\mathfrak{g} \curlywedge \mathfrak{g})
 \oplus  (\mathfrak{g}^{\ab} \wedge \mathfrak{g}^{\ab})\ar[r] & 0.
		}
		\]
\end{proposition}

\begin{proof} By the previous proposition it suffices to show that $\psi$ is injective.
By Proposition \ref{P:basis} one sees easily that the composition $\tau \circ \psi \colon \Gamma(\mathfrak{g}^{\ab}) \to \mathfrak{g}^{\ab} \otimes \mathfrak{g}^{\ab}$  maps  a basis of $\Gamma(\mathfrak{g}^{\ab})$ injectively into a set of linearly independent elements. Therefore $\tau \circ \psi$ is injective and $\psi$ is injective.
\end{proof}

\section{Comparison of the second Lie and Leibniz homologies of a Lie algebra}\label{S:h2}

In this section we return to the case when $\mathbb{K}$ is a field, $\mathfrak{g}$ denotes a Lie algebra and $H_2(\mathfrak{g})$ denotes  the second Chevalley-Eilenberg homology of $\mathfrak{g}$.
It is known that there is an epimorphism $t_{\mathfrak{g}} \colon HL_2(\mathfrak{g}) \to H_2(\mathfrak{g})$ defined in a natural way (see e.g. \cite{Pira} ).

\begin{proposition}
There  exists a vector subspace $V$  of $\Ker \{ t_{\mathfrak{g}} \colon HL_2(\mathfrak{g}) \to H_2(\mathfrak{g})\}$ such that we have an epimorphism $V  \to  \Gamma(\mathfrak{g}^{\ab})$. Hence,
if $\mathfrak{g}$ is not a perfect Lie algebra, then  $t_{\mathfrak{g}} \colon HL_2(\mathfrak{g}) \to H_2(\mathfrak{g})$
is not an isomorphism.
\end{proposition}
\begin{proof}  Let  $\mathfrak{g} \underset{Lie}{\star} \mathfrak{g}$ (resp. $\mathfrak{g} \underset{Lie}{ \curlywedge} \mathfrak{g}$) denote the non-abelian tensor (resp. exterior) square of the Lie algebra $\mathfrak{g}$ (see \cite{Ell91}).
Then we have two epimorphisms  $ \mathfrak{g} \curlywedge \mathfrak{g}\to \mathfrak{g} \underset{Lie}{ \star} \mathfrak{g}$ and $\mathfrak{g} \curlywedge \mathfrak{g}\to \mathfrak{g} \underset{Lie}{ \curlywedge} \mathfrak{g}$ defined in a natural way.
Since $t_{\mathfrak{g}} \colon HL_2(\mathfrak{g}) \to H_2(\mathfrak{g})$ can be viewed as the natural homomorphism
from  $\Ker \{\mathfrak{g} \curlywedge \mathfrak{g}\to \mathfrak{g} \}$ to $\Ker \{\mathfrak{g} \underset{Lie}{\curlywedge} \mathfrak{g} \to \mathfrak{g}\}$, we have that
$t_{\mathfrak{g}}(\Ker\{\mathfrak{g} \curlywedge \mathfrak{g}\to \mathfrak{g} \underset{Lie}{ \curlywedge} \mathfrak{g}\})=0$.
Let $V= \Ker\{\mathfrak{g} \curlywedge \mathfrak{g}\to \mathfrak{g} \underset{Lie}{ \curlywedge} \mathfrak{g}\}$. Consider the following commutative diagram with exact rows:
\begin{equation*}
\xymatrix@+20pt{
0\ \ar@{->}[r]& V \ \ar@{->}[r]
\ar@{->}[d]_{ }
 &\mathfrak{g}\curlywedge \mathfrak{g}\ar@{->}[r]
\ar@{->}[d]_{ }
&\mathfrak{g} \underset{Lie}{ \curlywedge} \mathfrak{g} \ar@{->}[r]
\ar@{=}[d]_{ }
&0 \\
0\ \ar@{->}[r]
& \Ker\{ \mathfrak{g} \underset{Lie}{ \star} \mathfrak{g}\to \mathfrak{g} \underset{Lie}{ \curlywedge} \mathfrak{g}\}   \ \ar@{->}[r]
 & \mathfrak{g} \underset{Lie}{ \star} \mathfrak{g}\ar@{->}[r]
& \mathfrak{g} \underset{Lie}{ \curlywedge} \mathfrak{g}\ar@{->}[r]
&0 .
}\end{equation*}
From this diagram we have an epimorphism $V\to  \Ker\{ \mathfrak{g} \underset{Lie}{ \star} \mathfrak{g}\to \mathfrak{g} \underset{Lie}{ \curlywedge} \mathfrak{g}\}$. Moreover, by \cite[Proposition 17]{Ell91}
 $\Ker\{ \mathfrak{g} \underset{Lie}{ \star} \mathfrak{g}\to \mathfrak{g} \underset{Lie}{ \curlywedge} \mathfrak{g}\}=\Gamma(\mathfrak{g}^{ab})$.
\end{proof}

\begin{proposition}  For a perfect Lie algebra $\mathfrak{g}$, we have the following exact sequence:
$$
0\to HL_2 (\mathfrak{g} \underset{Lie}{ \star} \mathfrak{g}) \to HL_2(\mathfrak{g}) \to H_2(\mathfrak{g}) \to 0 .
$$
\end{proposition}
\begin{proof} Let $\lam \colon \mathfrak{g} \star \mathfrak{g} \to \mathfrak{g} \underset{Lie}{ \star} \mathfrak{g}$ be the
natural epimorphism. We have the following commutative diagram:

\begin{equation*}
\xymatrix@+20pt{
0\ \ar@{->}[r]& HL_2(\mathfrak{g}) \ \ar@{->}[r]
\ar@{->}[d]_{ t_{\mathfrak{g}}}
 &\mathfrak{g}\curlywedge \mathfrak{g}\ar@{->}[r]
\ar@{->}[d]_{\lam }
& \mathfrak{g} \ar@{->}[r]
\ar@{=}[d]
&0 \\
0\ \ar@{->}[r]
& H_2(\mathfrak {g})  \ \ar@{->}[r]
 & \mathfrak{g} \underset{Lie}{ \star} \mathfrak{g}\ar@{->}[r]
& \mathfrak{g} \ar@{->}[r]
&0 ,
}\end{equation*}
where the first row is the universal central extension in the category of Lie algebras, and the second one is
 the universal central extension in the category of Leibniz algebras. This diagram implies that
$\Ker t_{\mathfrak{g}}=\Ker\lam$. Hence, $\Ker \lam$ is contained in the center of $\mathfrak{g}\curlywedge \mathfrak{g}$.
Moreover, since $\mathfrak{g}\curlywedge \mathfrak{g}$ is the supersolvable Leibniz algebra, we can conclude that
$\lam \colon \mathfrak{g} \star \mathfrak{g} \to \mathfrak{g} \underset{Lie}{ \star} \mathfrak{g}$ is the universal central
extension of $ \mathfrak{g} \underset{Lie}{ \star} \mathfrak{g}$ in the category of Leibniz algebras. Thus,
$$
HL_2(\mathfrak{g} \underset{Lie}{ \star} \mathfrak{g}) =\Ker \lam =\Ker t_{\mathfrak{g}}.
$$
The proof is complete.
\end{proof}

%\begin{proposition}
%	Let $\mathfrak{g}$ be a Leibniz algebra such that $\mathfrak{g}^{\ab}$ is free as a $\mathbb{K}$-module. Then there is an exact sequence of Leibniz algebras
%		\[\xymatrix{
%			0 \ar[r] & \Gamma(\mathfrak{g}^{\ab}) \ar[r]^-{\psi} & \mathfrak{g} \star \mathfrak{g} \ar[r]^-{\pi} & \mathfrak{g} \curlywedge \mathfrak{g} \ar[r] & 0.
%		}
%		\]
%\end{proposition}
%
%\begin{proof}
%	Let $\theta \colon \mathfrak{g} \star \mathfrak{g} \to \mathfrak{g}^{\ab} \star \mathfrak{g}^{\ab}$ be the canonical homomorphism. It is easy to see that $\mathfrak{g}^{\ab} \star \mathfrak{g}^{\ab}$ is isomorphic to the usual tensor product of modules $\mathfrak{g}^{\ab} \otimes \mathfrak{g}^{\ab}$. Then, the composition $\theta \circ \psi \colon \Gamma(\mathfrak{g}^{\ab}) \to \mathfrak{g}^{\ab} \star \mathfrak{g}^{\ab}$ is a homomorphism between free $\mathbb{K}$-modules and by Proposition \ref{P:basis} we know it maps linearly independent elements into linearly independent elements. Therefore $\theta \circ \psi$ is injective and $\psi$ is injective.
%\end{proof}

\section*{Acknowledgments}
The authors were supported by Ministerio de Econom\'ia y Competitividad (Spain) (European FEDER support included), grants MTM2013-43687-P and MTM2016-79661-P. The
second author was also supported by Xunta de Galicia, grant GRC2013-045 (European FEDER support included), by an FPU scholarship,
Ministerio de Educaci\'on, Cultura y Deporte (Spain) and by a Fundaci\'on Barri\'e scolarship. The third author was supported by Shota Rustaveli National Science Foundation, grant FR/189/5-113/14.

\end{document}